\documentclass[a4paper,10pt]{amsart}
\usepackage{amsmath,amssymb,amsxtra,amsfonts,amsthm,comment,url}
\usepackage{enumerate}
\usepackage[usenames]{color}
\usepackage{tikz}
\usetikzlibrary{arrows,backgrounds,decorations,automata}
\usepackage[unicode]{hyperref}
\usepackage{hypbmsec}
\theoremstyle{plain}
\newtheorem{theorem}{Theorem}
\newtheorem{lemma}{Lemma}
\newtheorem{proposition}{Proposition}
\theoremstyle{definition}

\newtheorem*{acknowledgements}{Acknowledgements}
\let\wh\widehat

\let\ge\geqslant
\let\le\leqslant
\newcommand{\trdeg}{\operatorname{tr\,deg}}
\renewcommand{\d}{{\mathrm d}}
\newcommand\bm{{\boldsymbol m}}
\newcommand\bn{{\boldsymbol n}}
\newcommand\bbf{{\boldsymbol f}}
\newcommand\bB{{\boldsymbol B}}
\newcommand\bM{{\bold M}}
\newcommand\bN{{\bold N}}
\newcommand\cM{{\mathcal M}}
\newcommand\cF{{\mathcal F}}

\newcommand\cP{{\mathcal P}}
\newcommand\SL{SL}
\newcommand{\raisecomma}{\raisebox{2pt}{$,$}}
\newcommand{\raisedot}{\raisebox{2pt}{$.$}}
\newcommand{\mutilde}{{\widetilde \mu}}
\newcommand{\Mtilde}{{\widetilde {\mathcal M}}}

\begin{document}

\title[Algebraic independence of Mahler functions]{Algebraic independence of Mahler functions\\ via radial asymptotics}

\author{Richard P.~Brent}
\address{Mathematical Sciences Institute\\
Australian National University\\
Canberra, ACT 0200\\
Australia}
\email{Richard.Brent@anu.edu.au}

\author{Michael Coons}
\address{School of Mathematical and Physical Sciences\\
The University of Newcastle\\
Callag\-han, NSW 2308\\
Australia}
\email{Michael.Coons@newcastle.edu.au}

\author{Wadim Zudilin}
\address{School of Mathematical and Physical Sciences\\
The University of Newcastle\\
Callag\-han, NSW 2308\\
Australia}
\email{Wadim.Zudilin@newcastle.edu.au}

\thanks{The research of R.\,P.~Brent was support by ARC grant DP140101417,
the research of M.~Coons was supported by ARC grant DE140100223, and
the research of W.~Zudilin was supported by ARC grant DP140101186}

\date{26 December 2014. \emph{Revised}: 4 April 2015}


\begin{abstract}
We present a new method for algebraic independence results in the context of Mahler's method.
In particular, our method uses the asymptotic behaviour of a Mahler function $f(z)$
as $z$ goes radially to a root of unity to deduce
algebraic independence results about the values of $f(z)$ at algebraic numbers.
We apply our method to the canonical example of a degree
two Mahler function; that is, we apply it to $F(z)$, the power series solution to
the functional equation $F(z)-(1+z+z^2)F(z^4)+z^4F(z^{16})=0$. Specifically, we
prove that the functions $F(z)$, $F(z^4)$, $F'(z)$, and $F'(z^4)$ are algebraically
independent over $\mathbb{C}(z)$. An application of a celebrated result of
Ku.~Nishioka then allows one to replace $\mathbb{C}(z)$ by $\mathbb{Q}$ when
evaluating these functions at a nonzero algebraic number $\alpha$ in the unit disc.
\end{abstract}

\maketitle

\section{Introduction}
\label{s1}

We say a function $f(z)\in\mathbb{C}[[z]]$ is a \emph{Mahler function} provided
there are integers $k\ge 2$ and $d\ge 0$ and polynomials $a(z),a_0(z),\dots,a_d(z)\in\mathbb{C}[z]$
with $a_0(z)a_d(z)\ne 0$ such that
\begin{equation}
\label{eq-mahler}
a(z)+a_0(z)f(z)+a_1(z)f(z^k)+\dots+a_d(z)f(z^{k^d})=0.
\end{equation}
We call the (minimal) integer $d$ the \emph{degree} of the Mahler function~$f$.
In the last few decades the study of Mahler functions has been given renewed importance
because of their relationships to theoretical computer science and linguistics~\cite{AS2003}.
In particular, the generating function of an automatic sequence is a Mahler function.

While transcendence questions concerning Mahler functions have more or less been
answered, much less is known about the deeper area of algebraic independence of the functions and their derivatives.
All of the current results in the latter direction, and certainly the most practical examples, concern only
Mahler functions of degree one \cite{B2012, B2013, BV2013}. Until now, these
results relied on a hypertranscendence criterion due to Ke.~Nishioka \cite{N1984}.
Recall that a function is called \emph{hypertranscendental} provided it
does not satisfy an algebraic differential equation; in other words, the function
and all its derivatives are algebraically independent over the field of
rational functions.

In this paper, we introduce a new method for proving algebraic independence
results for Mahler functions and their derivatives. We apply this method to a degree two Mahler
function introduced by Dilcher and Stolarsky \cite{DS2009}, which has quite
recently become the canonical example of a degree two Mahler function.
Specifically, we consider the function $F(z)\in\mathbb{Z}[[z]]$ satisfying the
functional equation
\begin{equation}
\label{Fdefn}
F(z)=(1+z+z^2)F(z^4)-z^4F(z^{16}),
\end{equation}
which starts
$$
F(z)=1+z+z^2+z^5+z^6+z^8+z^9+z^{10}+\cdots.
$$
Among various combinatorial properties, Dilcher and Stolarsky \cite{DS2009} showed
that all the coefficients of $F(z)$ are in $\{0,1\}$. Coons \cite{C2010} proved that $F(z)$ is transcendental
and Adamczewski \cite{A2010} gave the transcendence of the values $F(\alpha)$ for any nonzero algebraic number
$\alpha$ inside the unit disc. Recently, Bundschuh and V\"a\"an\"anen \cite{BV2014a}
proved that $F(z)$ and $F(z^4)$ are algebraically independent over $\mathbb{C}(z)$,
and very recently \cite{BV2014b} they showed that the functions $F(z)$, $F(z^2)$,
and $F(z^4)$ are algebraically independent over $\mathbb{C}(z)$.

Our central result is the following theorem.

\begin{theorem}
\label{funcs}
The functions $F(z)$, $F(z^4)$, $F'(z)$, and $F'(z^4)$ are algebraically independent over $\mathbb{C}(z)$.
\end{theorem}

\noindent
An application of Theorem~\ref{funcs} along with Mahler's powerful method implies the algebraic independence result for the values of the functions.

\begin{theorem}
\label{numbers}
For each non-zero algebraic number $\alpha$ inside the unit disc,
the numbers $F(\alpha)$, $F(\alpha^4)$, $F'(\alpha)$, and $F'(\alpha^4)$ are algebraically independent over $\mathbb{Q}$.
\end{theorem}

\noindent
Indeed, one expects the stronger version of algebraic independence of
the functions $F(z)$ and $F(z^4)$ along with {\em all} of their derivatives,
though the present methods seem inadequate for a result of this generality.

As alluded to in the above paragraphs, the novelty of our approach is the
avoidance of the hypertranscendence criterion of Ke.~Nishioka \cite{N1984}.
Ke.~Nishioka's criterion is very specialised and only applicable to Mahler
functions of degree one. In contrast, our method partly relies on understanding the
radial asymptotics of Mahler functions and can be applied to Mahler functions of
any degree. The analytical problem of determining this type of asymptotic
behaviour for Mahler functions is very classical, even for degree one Mahler
functions; e.g., see Mahler \cite{M1940}, de Bruijn \cite{dB1948}, Dumas \cite{D1993These},
and Dumas and Flajolet \cite{DF1996}. The importance of such asymptotics also appear
(though in a weaker form) in recent work of Adamczewski and Bell \cite{ABklMahler}.

In the case of $F(z)$, we prove the following result.

\begin{theorem}
\label{Fasymptotics}
As $z\to 1^-$, we have
$$
F(z)=\frac{C(z)}{(1-z)^{\lg\rho}}\cdot (1+O(1-z)),
$$
where $\lg$ denotes the base-$2$ logarithm, $\rho:=(1+\sqrt{5})/2$ denotes the golden ratio, and $C(z)$ is a positive
oscillatory term, which in the interval $(0,1)$ is bounded away from $0$ and $\infty$,
real analytic, and satisfies $C(z)=C(z^4)$.
\end{theorem}

This paper is organised as follows. In Section~\ref{sec:asymp} we prove Theorem~\ref{Fasymptotics} by
a careful study of the continued fraction for $F(z)/F(z^4)$.
In Section~\ref{s3} we use this knowledge to establish Theorem~\ref{funcs}:
assuming a polynomial relation in $F(z)$, $F(z^4)$, $F'(z)$, and $F'(z^4)$
the asymptotic behaviour of $F(z)$ as $z\to1^-$ allows us to significantly shorten it;
then using a linear algebra argument, we show that this reduced algebraic relation is not possible.
The related algebraic statement, Theorem~\ref{th-tensor}, is proved in generality in Section~\ref{s4}.
The derivation of Theorem~\ref{numbers} from Theorem~\ref{funcs} is
performed at the end of Section~\ref{s3}. Finally, in Section~\ref{s5} we discuss an
alternative proof of Theorem~\ref{Fasymptotics} that can be used in the asymptotical study of
general Mahler functions at arbitrary roots of unity.

We would like to point out that the methods of the paper apply
with no difficulty to the `satellite' function $G(z)$ (for definitions and related results
see Dilcher and Stolarsky \cite{DS2009}, Adamczewski \cite{A2010}, and Bundschuh and V\"a\"an\"anen \cite{BV2014a,BV2014b}),
so that all three theorems above remain true when we replace $F(z)$
in their statements with $G(z)$.

\section({A continued fraction related to \$F(z)\$ and asymptotics})%
{A continued fraction related to $F(z)$ and asymptotics}
\label{sec:asymp}

In this section, we prove Theorem \ref{Fasymptotics} as stated in the
introduction. In order to carry out our method, it is useful to define the
auxiliary function $\mu\colon[0,1)\to\mathbb{R}$ given by
\begin{equation}
\label{eq:mudef}
\mu(z) := \frac{F(z)}{F(z^4)}\,\raisedot
\end{equation}
{From}~\eqref{Fdefn} and~\eqref{eq:mudef}, $\mu(z)$ satisfies
the recurrence
\begin{equation}
\label{eq:murec}
\mu(z) = 1 + z + z^2 - \frac{z^4}{\mu(z^4)}\,\raisedot
\end{equation}
Our strategy is to analyse the asymptotic behaviour of $\mu(z)$ and then
deduce the corresponding behaviour of $F(z)$.

Note that $\mu(z)$ may be written as a continued fraction
\[
\mu(z) = 1 + z + z^2 - \dfrac{z^4}{1 + z^4 + z^{2\cdot 4} - \dfrac{z^{4^2}}{1 + z^{4^2} + z^{2\cdot4^2} - {\atop\ddots}}}.
\]
Also, from~\eqref{eq:mudef}, $F(z)$ is given by the infinite product
\begin{equation*}
F(z) = \prod_{k=0}^\infty \mu(z^{4^k}).
\end{equation*}
In this sense we have an `explicit solution' for $F(z)$ as an infinite product of continued fractions.

Before continuing, we make some remarks on notation. Since logarithms to different bases
occur naturally in the analysis, we write $\ln x$ for the natural logarithm and $\lg x$ for the logarithm to the base~$2$.
As in the statement of Theorem~\ref{Fasymptotics}, we define $\rho := (1+\sqrt{5})/2 \approx 1.618$ to be the golden ratio, and note
that $\rho^2 = \rho + 1$.

The following few lemmas provide the needed background for the proof of Theorem~\ref{Fasymptotics}
concerning the asymptotics of $F(z)$ as $z\to 1^-$.

\begin{lemma}
\label{lemma:Fproperties}
The power series \[F(z) = \sum_{n=0}^\infty c_n z^n\]
has coefficients $c_n \in \{0,1\}$.  Also,
$F(z)$ is strictly monotone increasing and unbounded for $z \in [0,1)$,
and cannot be analytically continued past the unit circle.
\end{lemma}

\begin{proof}
Since the coefficients $c_n$ are in $\{0,1\}$ (see \cite{BV2014a,DS2009}) and
infinitely many are nonzero, the strict monotonicity and unboundedness of
$F(z)$ follow easily. Thus, $F(z)$ has a singularity at $z=1$.

{From} the functional equation~\eqref{Fdefn} it follows that $F(z)$ has a
singularity at $z = e^{2\pi i/2^k}$ for all nonnegative integers $k$. Thus,
there is a dense set of singularities on the unit circle, so the unit circle is a
natural boundary. See also Bundschuh and V\"a\"an\"anen \cite[Theorem 1.1]{BV2014a}
for a proof of the last part of the lemma.
\end{proof}

\begin{lemma}
\label{lemma:special_values}
If $z\in[0,1)$, then $\mu(z)\ge 1$. Moreover, if $\mu_1 := \lim_{z \to 1^-}\mu(z)$
and $\mu_1' := \lim_{z \to 1^-}\mu'(z)$, then
\begin{equation*}
\mu_1 = \frac{3+\sqrt{5}}{2} = \rho^2 \approx 2.618
\qquad\text{and}\qquad
\mu_1' = \frac{21+8\sqrt{5}}{11} \approx 3.535.
\end{equation*}
\end{lemma}

\begin{proof}
Suppose that $z \in [0,1)$.
Since $F(z)$ is monotonic increasing on $[0,1)$,
we have $F(z) \ge F(z^4) \ge1$, so $\mu(z) \ge1$.

Define
\[
Q(z) := \frac{1+z+z^2 + \sqrt{(1+z+z^2)^2-4z^4}}{2}
\]
to be the larger root of
\[Q(z) = 1+z+z^2 - \frac{z^4}{Q(z)}\,\raisedot\]
Observe that $Q(z)$ is a continuous monotone increasing function on $[0,1]$,
and $Q(1) = (3+\sqrt{5})/2$.

Take an arbitrary $z_0 \in (0,1)$, and define $z_k := z_0^{1/4^k}$,
so $z_{k-1} = z_k^4$ for $k \ge1$ and $\lim_{k\to\infty} z_k = 1$.
For notational convenience, we also define
$y_k := \mu(z_k)$ and $Q_k := Q(z_k)$; in particular,
\[
Q_k = 1 + z_k + z_k^2 - z_k^4/Q_k
\]
and
\[
y_k = 1 + z_k + z_k^2 - z_k^4/y_{k-1}
\]
from the functional equation~\eqref{eq:murec}.
Since $\lim_{k\to\infty}Q_k = Q(1) > 2$,
we can assume that $k_0\ge1$ is sufficiently large that $Q_k \ge 2$ for all $k\ge k_0$.
Thus
\begin{align*}
|Q_k-y_k|
&= |z_k^4(y_{k-1}^{-1} - Q_k^{-1})|
= \biggl|\frac{z_k^4(Q_k-y_{k-1})}{Q_ky_{k-1}}\biggr|
\\
&\le\frac{|Q_k-y_{k-1}|}{2}
\le\frac{|Q_{k-1}-y_{k-1}|}{2}+\frac{|Q_k-Q_{k-1}|}{2}\,\raisecomma
\end{align*}
using $|z_k| \le 1$, $y_{k-1} \ge1$, $|Q_k| \ge2$, and the triangle inequality.
It follows from $\lim_{k\to\infty}(Q_k-Q_{k-1}) = 0$ that $\lim_{k\to\infty}(Q_k-y_k) = 0$.
Thus $\lim_{k\to\infty}y_k = Q(1)$, which completes the proof of $\mu_1=\rho^2$.

Differentiating each side of the recurrence~\eqref{eq:murec} gives
\begin{equation}
\label{eq:muprime1}
\mu'(z) = 1 + 2z - \frac{4z^3}{\mu(z^4)} + \frac{4z^7\mu'(z^4)}{\mu(z^4)^2}.
\end{equation}
As $z\to 1^-$,
\[
1+2z-\frac{4z^3}{\mu(z^4)} \to 3-\frac4{\mu_1}
\qquad\text{and}\qquad
\frac{4z^7}{\mu(z^4)^2} \to \frac4{\mu_1^2},
\]
so~\eqref{eq:muprime1} may be written as
\begin{equation*}
\mu'(z) = 3 - 4/\mu_1 + o(1) + (4/\mu_1^2 + o(1))\mu'(z^4).
\end{equation*}
Using the latter expression,
it can be shown that $\mu'(z) \to \mu_1'$, where $\mu_1'$ satisfies
\[
\mu_1' = 3 - 4/\mu_1 + (4/\mu_1^2)\mu_1',
\]
so
\[
\mu_1' = \frac{3-4/\mu_1}{1-4/\mu_1^2} = \frac{21 + 8\sqrt{5}}{11}\,\raisedot
\]
We omit the details, but note that $|4/\mu_1^2| < 1$, so the iteration
\[
m_k = 3 - 4/\mu_1 + (4/\mu_1)^2 m_{k-1}
\]
converges, and $\lim_{k\to\infty} m_k = \mu_1'$.
\end{proof}

In view of Lemma~\ref{lemma:special_values}, we
define by continuity $\mu(1) := \mu_1$ and $\mu'(1) := \mu_1'$.
Since $\mu''(z)$ is unbounded as $z \to 1^-$, this
process cannot be continued; see Figure \ref{mugraph} for a graph of $\mu(z)$ and
$\mu'(z)$ for $z\in[0,1)$.

\begin{figure}[htdp]
\includegraphics[width=4.7in,height=1.7in]{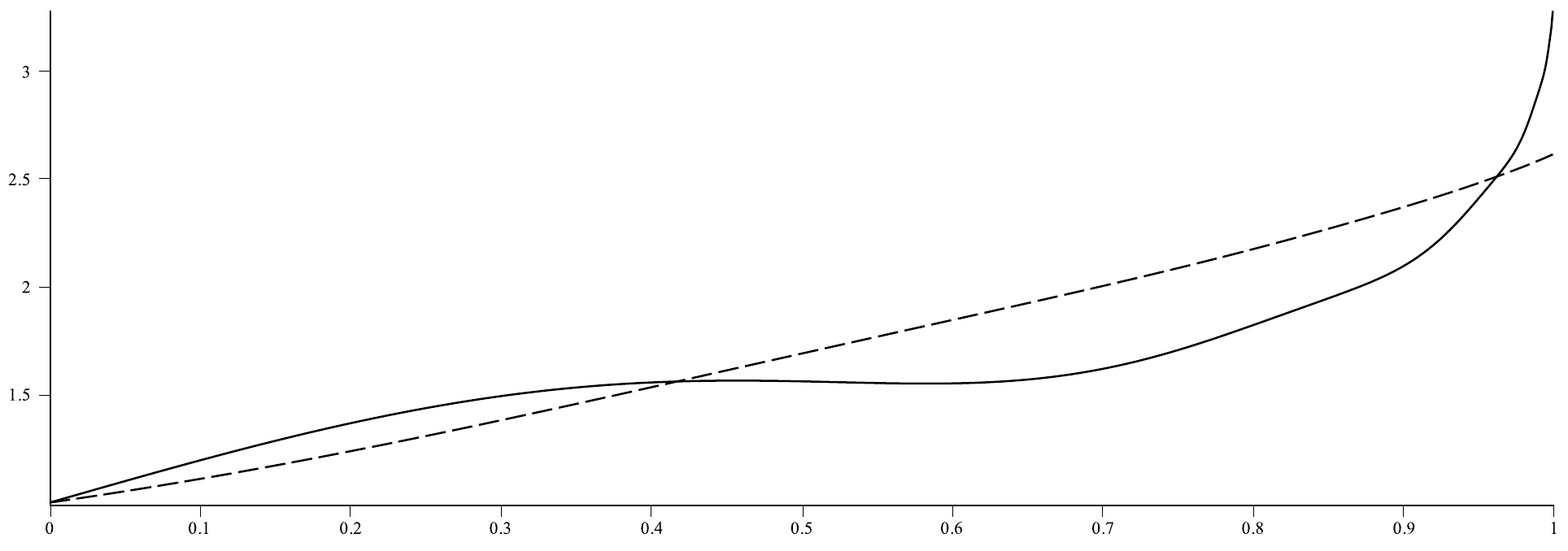}
\caption{The functions $\mu(z)$ (dashed) and $\mu'(z)$ (solid) for $z\in[0,1)$.}
\label{mugraph}
\end{figure}

\begin{lemma}
\label{lemma:mu_approx}
Let $\alpha$ be any constant satisfying $\alpha < 2\lg\rho \approx 1.388$.
Then, for $t \in (0,\infty)$, we have
\begin{equation}
\label{eq:muppineq1}
\mu''(e^{-t}) = O(t^{\alpha-2})
\end{equation}
and
\begin{equation}
\label{eq:muineq}
\mu(e^{-t}) = \mu_1 - t\mu'_1 + O(t^{\alpha}).
\end{equation}
\end{lemma}

\begin{proof}
Let $z = e^{-t} \in (0,1)$.
Differentiating both sides of~\eqref{eq:muprime1} with respect to~$z$ gives
\begin{equation}
\label{eq:mupp}
\mu''(e^{-t}) = A(t) + B(t)\mu''(e^{-4t}),
\end{equation}
where $A(t)$ is uniformly bounded, say $|A(t)| \le A$,
and
\begin{equation}
\label{eq:Bt}
B(t) = \frac{16e^{-10t}}{\mu(e^{-4t})^2} = \frac{16}{\mu_1^2} + O(t)
\end{equation}
as $t \to 0^+$.

We now prove by induction on $k \ge 0$ that, if $t_0$ is sufficiently small,
$t_k := t_0/4^k$, and $C$ is sufficiently large, then
\begin{equation}
\label{eq:mupp_induction}
\mu''(e^{-t_k}) < Ct_k^{\alpha-2}
\end{equation}
holds for all $k\ge 0$.

By the choice of $\alpha$, we have
$$
\delta := 1 - \frac{16}{\mu_1^2}\cdot 4^{\alpha-2} > 0.
$$
By~\eqref{eq:Bt}, there exists $\varepsilon > 0$ such that,
for all $t \in (0,\varepsilon)$, we have
$B(t) < (16/\mu_1^2)(1+\delta)$.  Thus, for all $t \in (0,\varepsilon)$,
\[
4^{\alpha-2}B(t) < (1-\delta)(1+\delta) = 1-\delta^2.
\]
For an arbitrary $t_0 \in (0,\varepsilon)$, choose
\begin{equation}
\label{eq:K}
C > \max\{\mu''(e^{-t_0})/t_0^{\alpha-2},	A/\delta^2\}.
\end{equation}
Thus $\mu''(e^{-t_0}) < Ct_0^{\alpha-2}$, so the
inductive hypothesis~\eqref{eq:mupp_induction} holds for $k=0$.
Suppose that it holds for some $k \ge 0$.
Then from~\eqref{eq:mupp},
\begin{align*}
\mu''(e^{-t_{k+1}})
&= \mu''(e^{-t_k/4})
\le A(t_k/4) + B(t_k/4)\mu''(e^{-t_k})
\\
&< A + 4^{2-\alpha}(1-\delta^2)Ct_k^{\alpha-2}
= A + C(1-\delta^2)t_{k+1}^{\alpha-2}
\\
&= (A - C\delta^2t_{k+1}^{\alpha-2}) + Ct_{k+1}^{\alpha-2}
< Ct_{k+1}^{\alpha-2},
\end{align*}
where on the final step we used $A<C\delta^2<C\delta^2t_{k+1}^{\alpha-2}$, by the choice~\eqref{eq:K} of~$C$
and also since $t_{k+1} \in (0,1)$ and $\alpha < 2$.
Thus, \eqref{eq:mupp_induction} holds for all $k \ge 0$, by induction.
This proves~\eqref{eq:muppineq1}.
To prove~\eqref{eq:muineq} we integrate twice over the interval $[0,t]$.
\end{proof}

With a similar (but more precise) proof, we can show that
the bounds~\eqref{eq:muppineq1} and \eqref{eq:muineq} of
Lemma~\ref{lemma:mu_approx} hold for $\alpha = 2\lg\rho$.
We omit the details since this result is not necessary in what follows.
Numerical experiments indicate that the constant $2\lg\rho$
is best possible~-- see Table~$\ref{tab:mu}$, where the last column
gives $(\mu(e^{-t}) - (\mu_1 - t\mu_1'))/t^{2\lg\rho}$. Observe the
small oscillations in the last column (these are discussed
at the end of this section).

\begin{table}[htdp] 	
\begin{center}
\caption{Approximation of $\mu(e^{-t})$ for $t = 2^{-k}$, $20\le k\le24$,
where $e_1(t) := \mu(e^{-t}) - (\mu_1 - t\mu_1')$}
\begin{tabular}{ccccc}
$k$  & $t=2^{-k}$   & $\mu(e^{-t})$ & $e_1(t)$ & $e_1(t)/t^{2\lg\rho}$ \\
\hline
20 & $9.5367\cdot10^{-7}$ & 2.6180306 & $1.1708\cdot10^{-8}$  & 2.6790\vphantom{$\big|^0$} \\
21 & $4.7684\cdot10^{-7}$ & 2.6180323 & $4.4999\cdot10^{-9}$  & 2.6958 \\
22 & $2.3842\cdot10^{-7}$ & 2.6180331 & $1.7079\cdot10^{-9}$  & 2.6787 \\
23 & $1.1921\cdot10^{-7}$ & 2.6180336 & $6.5648\cdot10^{-10}$ & 2.6956 \\
24 & $5.9605\cdot10^{-8}$ & 2.6180338 & $2.4917\cdot10^{-10}$ & 2.6786 \\
\hline
\end{tabular}
\label{tab:mu}
\end{center}
\end{table}

The following lemma is not necessary in what follows, but we state it here for
its independent interest and provide a sketch of the proof;
check also with Figure \ref{mugraph} for a graph of $\mu(z)$.

\begin{lemma}
\label{lemma:mu_monotonic}
The function $\mu(z)$ is strictly monotone increasing for $z\in[0,1)$.
\end{lemma}

\begin{proof}[Sketch of proof]
Suppose that $z\in[0,1)$ and $N\ge1$.
Since $F(z) = \sum_{n\ge0} c_n z^n$, where the $c_n \in \{0,1\}$, we can
bound the `tails'
\[
\sum_{n\ge N} c_n z^n \le\frac{z^N}{1-z}
\qquad\text{and}\qquad
\sum_{n\ge N}n c_n z^{n-1} \le\frac{Nz^{N-1}}{1-z} + \frac{z^N}{(1-z)^2} \,\raisedot
\]
Thus, given $z_0<1$ and $\varepsilon > 0$,
we can easily find $N = N(\varepsilon)$ such that, for all $z\in [0,z_0]$,
\[
0 \le F(z) - F_N(z) \le\varepsilon
\qquad\text{and}\qquad
0 \le F'(z) - F'_N(z) \le\varepsilon,
\]
where $F_N(z) := \sum_{n=0}^{N-1} c_nz^n$ is the truncated power series approximating $F(z)$.

{From}~\eqref{eq:mudef} we have $\mu(z) \ge1$ and
\[
\frac{\mu'(z)}{\mu(z)} = \frac{F'(z)}{F(z)} -
4z^3\frac{F'(z^4)}{F(z^4)}\,\raisedot
\]
Take $z_0 = 3/4$.  Using the above and a rigorous numerical computation,
we can show that $\mu'(z) > 0$ for $z\in [0,z_0]$,
and also that $\mu(z_0^4) > 4/3$.
Thus, for $z \in [z_0,z_0^{1/4}]$, we have $\mu(z^4) > 4/3$.
In particular, $\mu(z_0) > 4/3$.

Define $z_k := (3/4)^{1/4^k}$, so $z_{k+1} = z_k^{1/4}$.
We prove, by induction on $k\ge0$, that
$\mu'(z) > 0$ for $z \in [0,z_k]$.  The base case ($k=0$) has been
established.  Assume that the result holds for $k \le K$;
hence, for $z\in [z_K,z_{K+1}]$, we have $\mu(z^4) > 4/3$.
Now, from~\eqref{eq:muprime1},
\[\mu'(z) \ge1 + 2z - \frac{4z^3}{\mu(z^4)}
 \ge1 + 2z - 3z^3 = (1-z)(1+3z+3z^2) > 0.\]
In other words, the result holds for $k = K+1$,
thus it holds for all $k\ge0$, by induction.
Since $\lim_{k\to\infty}z_k = 1$, this completes the proof.
\end{proof}

We are now in a position to treat the asymptotics of $F(z)$ as $z\to 1^-$. To
this end, we define the Mellin transforms
\begin{equation}
\label{eq:calF}
\cF(s) := \int_0^\infty \ln F(e^{-t})\,t^{s-1}\,\d t
\end{equation}
and
\begin{equation}
\label{eq:calM}
\cM(s) := \int_0^\infty \ln\mu(e^{-t})\,t^{s-1}\,\d t
\end{equation}
where the integrals converge for $\Re(s) > 0$, and by analytic
continuation elsewhere.
{From}~\eqref{eq:mudef} and well-known properties of Mellin
transforms (see, for example, \cite[Appendix~B.7]{FS}), we have
\begin{equation}
\label{eq:FMreln}
(1-4^{-s})\cF(s) = \cM(s).
\end{equation}

We deduce the asymptotic behaviour of $F(e^{-t})$ for small
positive~$t$ from knowledge of the singularities of $\cF(s)$.
Before doing this, we use analytic continuation to extend
the definitions \eqref{eq:calF} and \eqref{eq:calM} into the left half-plane.

Define
\[
\mutilde(t) := \ln\mu(e^{-t}) - \ln(\mu_1)e^{-\lambda t},
\]
where
\[
\lambda := \frac{\mu_1'}{\mu_1\ln\mu_1} \approx 1.403
\]
is a positive constant; the reason for our choice of $\lambda$ will
soon be clear.

Clearly $\mutilde(t) = O(e^{-t})$ as $t\to +\infty$.  Also,
from Lemma~\ref{lemma:mu_approx}, as $t \to 0^+$ we have for any constant $\alpha < 2\lg\rho \approx 1.388$
\[
\mutilde(t) = (\lambda\ln\mu_1 - \mu_1'/\mu_1)t + O(t^\alpha)=O(t^\alpha),
\]
by our choice of $\lambda$.

{From}~\eqref{eq:calM} and the definition of $\mutilde(t)$, we have
\begin{equation}
\label{eq:calM3}
\cM(s) = \Mtilde(s) + \ln(\mu_1)\lambda^{-s}\Gamma(s),
\end{equation}
where
\begin{equation}
\label{eq:Mtilde3}
\Mtilde(s) := \int_0^\infty \mutilde(t)t^{s-1}\,\d t.
\end{equation}
However, the integral in~\eqref{eq:Mtilde3} converges
for $\Re(s) > -\alpha$.  Since $\alpha$ may be chosen
arbitrarily close to $2\lg\rho$, this implies that
\eqref{eq:calM3} and \eqref{eq:Mtilde3} give the analytic continuation of $
\cM(s)$ into a meromorphic function in the half-plane
${\mathcal H} := \{s\in\mathbb{C}: \Re(s) > -2\lg\rho\}$.

Since $\Mtilde(s)$ has no singularities in $\mathcal H$, it follows
from~\eqref{eq:calM3} that the singularities of $\cM(s)$
in $\mathcal H$ are precisely those of $\ln(\mu_1)\lambda^{-s}\Gamma(s)$.
Also, from~\eqref{eq:FMreln}, the singularities of $\cF(s)$
in $\mathcal H$ are precisely those of $\cM(s)/(1-4^{-s})$.
We conclude that the Mellin transform $\cF(s)$
has three types of singularities in $\mathcal H$, as follows:
\begin{enumerate}
\item[(a)]
a double pole at $s = 0$, since $\Gamma(s)$ has a pole there, and
the denominator $1-4^{-s}$ vanishes at $s=0$;
\item[(b)]
simple poles at $s = ik\pi/\ln 2$ for $k \in \mathbb{Z}\setminus\{0\}$,
since the denominator $1-4^{-s}$ vanishes at these points; \\
and
\item[(c)]
a simple pole at $s=-1$, since $\Gamma(s)$ has a pole there.
\end{enumerate}
We are now ready to prove the following result, which gives
the asymptotic behaviour of $F(z)$ as $z \to 1^-$.  It is convenient
to express the result in terms of $\ln F(e^{-t})$.
Indeed, Theorem~\ref{Fasymptotics} is a weaker result, written in terms of~$F(z)$, of the following statement.

\begin{proposition}
\label{thm:lnFapprox2}
For small positive~$t$,
\begin{equation}
\label{eq:Fapprox2}
\ln F(e^{-t}) = -\lg\rho\cdot\ln t + c_0 + \sum_{k=1}^\infty a_k(t)	+ c_1 t + O(t^{\alpha}),
\end{equation}
where $c_0$ is given by~\eqref{eq:c0a},
$c_1$ is given by~\eqref{eq:c1},
$\alpha < 2\lg\rho \approx 1.388$,
and
\[
a_k(t) = \frac{1}{\ln 2}\,\Re\biggl(\cM\biggl(\frac{ik\pi}{\ln 2}\biggr)\exp(-ik\pi\lg t)\biggr).
\]
\end{proposition}

\begin{proof}
We consider the three types of singularities of $\cF(s)$ in $\mathcal H$.
For case~(a), the double pole at $s=0$, we need the first two terms in the
Laurent expansion of $\cF(s)$. It is convenient to define\footnote{The
reader may think of $L(s)$ as the Dirichlet series $\sum_{n=1}^\infty b_n n^{-s}$,
where the $b_n$ are defined to be the coefficients in the power series
$\ln \mu(z) = \sum_{n=1}^\infty b_n z^n$. Be warned that $\mu(z)$ has a zero
at $z_0 \approx -0.2787 + 0.7477i$, so the power series has
radius of convergence $R = |z_0| \approx 0.7979 < 1$.  Thus, the $b_n$ have
faster than polynomial growth, and the Dirichlet series does not converge anywhere.
}
\[
L(s) := \frac{\cM(s)}{\Gamma(s)}\,\raisecomma
\]
so, from~\eqref{eq:calM3},
\begin{equation}
\label{eq:Ls2}
L(s) = \frac{\Mtilde(s)}{\Gamma(s)} + \ln(\mu_1)\lambda^{-s}.
\end{equation}
Taking the limit as $s\to 0$ in~\eqref{eq:Ls2} gives
\[
L(0) = \ln\mu_1 = 2\ln\rho \approx 0.9624.
\]
Differentiating both sides of~\eqref{eq:Ls2} and then taking the limit
as $s \to 0$ gives
\[
L'(0) = \Mtilde(0) - 2\ln\lambda\cdot\ln\rho \approx 0.05706.
\]
Near $s=0$ we have
\begin{gather*}
L(s) = L(0)\biggl(1 + \frac{L'(0)}{L(0)}s + O(s^2)\biggr),
\qquad
\Gamma(s) = \frac{1}{s}(1 - \gamma s + O(s^2)),
\\ \intertext{and}
(1-4^{-s})^{-1} = \frac{1}{2s\ln 2}(1+s\ln 2 + O(s^2)),
\end{gather*}
so
\[
\cF(s) = \frac{L(0)}{2\ln 2}\cdot\frac{1}{s^{2}} + \frac{c_0}{s} + O(1),
\]
where $L(0)/\ln 4 = \lg\rho$ and
\begin{equation}
\label{eq:c0a}
c_0 = \frac{(\ln 2 - \gamma)L(0) + L'(0)}{2\ln 2}
\approx 0.1216438693\,.
\end{equation}
Now, the `Mellin dictionary' of~\cite[pg.~765]{FS} shows that the double pole
at $s=0$ contributes the two leading terms $-\lg\rho\cdot\ln t + c_0$
of~\eqref{eq:Fapprox2}.

For case~(b), the poles at $s = ik\pi/\ln 2$ for $k \in \mathbb{Z}\setminus\{0\}$
are simple and have residue $\cM(ik\pi/\ln 2)/\ln 4$. Thus from the simple pole
at $ik\pi/\ln 2$ we get a term
\[
T_k(t) := \frac{1}{\ln 4}
\cM\left(\frac{ik\pi}{\ln 2}\right)\exp(-ik\pi\lg t).
\]
Combining the terms $T_k(t)$ and $T_{-k}(t)$ for $k\ge1$,
the imaginary parts cancel and we are left with the oscillatory term $a_k(t)$
in~\eqref{eq:Fapprox2}. Of course, in order to write the infinite sum over the $a_k(t)$ as stated in the proposition, we must show that this sum converges. Note that $\mutilde''(t) = O(t^{\alpha-2})$ as $t \to 0^+$, and
$\mutilde''(t)$ decreases exponentially as $t \to +\infty$.
Suppose $y \in \mathbb{R}\backslash\{0\}$. Then, using integration by parts once, we have
$$
\Mtilde(iy)
 = \int_0^\infty \mutilde(t)t^{iy-1}\,{\rm d}t= \left[\mutilde(t)\,\frac{t^{iy}}{iy}
    - \int\mutilde'(t)\,\frac{t^{iy}}{iy}\,{\rm d}t\right]_0^\infty= -\,\frac{1}{iy}\int_0^\infty\mutilde'(t)\,t^{iy}\,{\rm d}t,
$$
and twice, we have
$$
\Mtilde(iy)= \frac{1}{iy(iy+1)}\int_0^\infty \mutilde''(t)\,t^{iy+1}\,{\rm d}t [2pt]
= \frac{I_0(y)+I_1(y)}{iy(iy+1)}\,\raisecomma
$$
where
\[I_0(y) = \int_0^1 \mutilde''(t)\,t^{iy+1}\,{\rm d}t \;\text{ and }\;
  I_1(y) = \int_1^\infty \mutilde''(t)\,t^{iy+1}\,{\rm d}t.\]
Now, using the asymptotic bounds on $\mutilde''(t)$ for small and large $t$, respectively, we have
\[|I_0(y)| \ll \int_0^1 t^{\alpha-1}\,{\rm d}t \ll 1\;
\text{ and }\;
|I_1(y)| \ll \int_1^\infty te^{-t}\,{\rm d}t \ll 1,\]
so as $y\to\infty$,
\[\Mtilde(iy) \ll |y|^{-2}.\]
Also, it follows from the complex version of Stirling's formula that
$\Gamma(iy) \ll e^{-\pi y/2}$,
so $\mathcal{M}(iy) \ll |y|^{-2}$ as $y \to \infty$.
Thus,
\[\sum_{k=1}^\infty |\mathcal{M}(ik\pi/\ln(2))| \ll
  \sum_{k=1}^\infty k^{-2} < \infty,\]
and the series $\sum_{k=1}^\infty a_k(t)$ is uniformly and absolutely
convergent.

For case~(c), the factor $(1-4^{-s})^{-1}$ is $-1/3$ at $s=-1$,
so $\cF(s)$ has a pole with residue
\begin{equation}
\label{eq:c1}
c_1 = \frac{\lambda\ln\mu_1}{3}
      = \frac{\mu_1'}{3\mu_1}
      = \frac{23+3\sqrt{5}}{66} \approx 0.4501
\end{equation}
at $s=-1$. This accounts for the term $c_1 t$ in~\eqref{eq:Fapprox2}.

Finally, the error term $O(t^\alpha)$ in~\eqref{eq:Fapprox2}
follows from the fact that we have only considered the singularities
of $\cF(s)$ in $\mathcal H$.
\end{proof}

We may write $a_k(t)$ as
\[
a_k(t) = A_k\cos(k\pi\lg t) + B_k\sin(k\pi\lg t).
\]
Define $C_k := \sqrt{A_k^2 + B_k^2} = \max_{t>0} |a_k(t)|$.
The constants $A_k$, $B_k$ and $C_k$ for $k \le4$ are given in Table~$\ref{tab:ABC}$.
We discuss the methods used to compute the numerical values of these constants at
the end of this section.

\begin{table}[htdp]
\begin{center}
\caption{The constants $A_k$, $B_k$ and $C_k$ related to $a_k(t)$ for $k\le4$.}
\begin{tabular}{cccc}
$k$ & $A_k$ & $B_k$ & $C_k$ \\
\hline
1 & $+2.009968436\cdot10^{-3}$  & $-6.155485619\cdot10^{-4}$  & $2.102111592\cdot10^{-3}$\vphantom{$\big|^0$} \\
2 & $-1.751530562\cdot10^{-6}$  & $+1.354122041\cdot10^{-6}$  & $2.213934464\cdot10^{-6}$  \\
3 & $+4.561611933\cdot10^{-10}$ & $-2.802129666\cdot10^{-9}$  & $2.839016326\cdot10^{-9}$  \\
4 & $+2.421586941\cdot10^{-13}$ & $+3.247722091\cdot10^{-12}$ & $3.256737573\cdot10^{-12}$ \\
\hline
\end{tabular}
\label{tab:ABC}
\end{center}
\end{table}

\begin{proof}[Proof of Theorem~\textup{\ref{Fasymptotics}}]
If we define $C(z) := (1-z)^{\lg \rho}F(z)$, then clearly
$C(z)$ is positive for $z \in [0,1)$, and $C(0) = 1$.
Also, for small positive $t$,
Proposition~\ref{thm:lnFapprox2} gives
\[
C(e^{-t}) = D(t)e^{O(t)},
\]
where
\[
D(t) = \exp\biggl(c_0 + \sum_{k=1}^\infty a_k(t)\biggr)
\]
is a continuous function, which is periodic in the variable $\lg t$.
Since $F(e^{-t}) > 1$ for $t \in (0,\infty)$, we must have
$$
0 < \inf_{t>0} D(t) \le\sup_{t>0}D(t) < \infty.
$$
Thus
$$
0 < \inf_{0<t<1} C(e^{-t}) < \sup_{0<t<1}C(e^{-t}) < \infty.
$$
However, it is easy to see directly that $C(e^{-t})$ is bounded
away from zero and infinity for $t \in [1,\infty)$.
\end{proof}

In the remainder of this section, we briefly discuss some of the numerical
findings and computations that were used throughout this section.

Regarding the function $C(z)$ of Theorem~\ref{Fasymptotics}, we
find numerically that
$1 < C(z) < 1.14$ for all $z \in (0,1)$, and
$1.11 < C(z) < 1.14$ for all $z \in [1/2,1)$.

In order to evaluate
$\cM(\pi ik/\ln2)$ for $k \in \mathbb{Z}\setminus\{0\}$, by \eqref{eq:calM3}
it suffices to evaluate $\Mtilde(\pi ik/\ln2)$, since the term involving the $\Gamma$-function can be evaluated by standard methods.
For purposes of numerical computation,
we transform the integral~\eqref{eq:Mtilde3} as follows.

Changing variables $t = e^u$, we have
\[
\Mtilde\biggl(\frac{\pi ik}{\ln 2}\biggr)
= \int_{-\infty}^{+\infty} \mutilde(e^u)e^{\pi iku/\ln 2}\,\d u.
\]
Now let $v := ku/(2\ln 2)$, so
\[
\Mtilde\biggl(\frac{\pi ik}{\ln 2}\biggr)
= \frac{2\ln 2}{k} \int_{-\infty}^{+\infty} \mutilde(e^{2\ln(2)v/k}) e^{2\pi iv}\,\d v.
\]
Using the $1$-periodicity of $e^{2\pi iv}$, we obtain
\begin{equation}
\label{eq:Mtilde-periodic}
\Mtilde\biggl(\frac{\pi ik}{\ln 2}\biggr)
= \frac{2\ln 2}{k} \int_0^1 f_k(v)e^{2\pi iv}\,\d v,
\end{equation}
where $f_k(v)$ is a $1$-periodic function defined by a rapidly convergent series;
\[
f_k(v) := \sum_{j\in\mathbb{Z}} \mutilde(e^{2\ln(2)(v+j)/k}).
\]

The integral in~\eqref{eq:Mtilde-periodic} can be evaluated by any method
which is suitable for periodic integrands (a simple and good choice is the
trapezoidal rule~\cite{TW}).

\begin{table}[htdp]	
\begin{center}
\caption{Approximation of $\ln(F(e^{-t}))$ using Proposition~\ref{thm:lnFapprox2}.
Here $e_2(t)$ is defined as the approximation given by~\eqref{eq:Fapprox2} minus the exact value $\ln F(e^{-t})$.}
\begin{tabular}{cccc}
$t$ & $\ln(F(e^{-t}))$ & $e_2(t)$ & $e_2(t)/t^{2\lg\rho}$ \\
\hline
$1.0\cdot10^{-1}$  & 1.756508934 & $7.11\cdot10^{-3}$  & 0.1739\vphantom{$\big|^0$} \\
$1.0\cdot10^{-2}$  & 3.322632048 & $2.93\cdot10^{-4}$  & 0.1755 \\
$1.0\cdot10^{-3}$  & 4.919666200 & $1.19\cdot10^{-5}$  & 0.1748 \\
$1.0\cdot10^{-4}$  & 6.514164850 & $4.91\cdot10^{-7}$  & 0.1757 \\
$1.0\cdot10^{-5}$  & 8.114306645 & $2.00\cdot10^{-8}$  & 0.1755 \\
$1.0\cdot10^{-6}$  & 9.714782160 & $8.16\cdot10^{-10}$ & 0.1748 \\
$1.0\cdot10^{-7}$  & 11.30965459 & $3.35\cdot10^{-11}$ & 0.1757 \\
$1.0\cdot10^{-8}$  & 12.91018122 & $1.37\cdot10^{-12}$ & 0.1755 \\
$1.0\cdot10^{-9}$  & 14.51031430 & $5.57\cdot10^{-14}$ & 0.1748 \\
$1.0\cdot10^{-10}$ & 16.10521012 & $2.29\cdot10^{-15}$ & 0.1757 \\
\hline
\end{tabular}
\label{tab:lnF}
\end{center}
\end{table}

Table~\ref{tab:lnF} shows the results of a numerical computation using
Proposition~\ref{thm:lnFapprox2}. We used $8$ terms in the sum
over $a_k(t)$.  In the table, $e_2(t)$
is defined as the approximation given by~\eqref{eq:Fapprox2} minus the exact
value $\ln F(e^{-t})$. It appears from the last column of the table
that the error is of order $t^{2\lg\rho}$. Also, the last column does not
appear to tend to a limit as $t \to 0^+$; instead it
fluctuates in a small interval. The same phenomenon may be observed in the
last column of Table~\ref{tab:mu}.  This suggests that $\cM(s)$
and $\cF(s)$ have poles at $s = -2\lg\rho + ik\pi/\ln 2$
for $k \in \mathbb{Z}$, as expected from the form of~\eqref{eq:mupp}.

\section({Algebraic independence of \$F(z)\$, \$F(z\000\1364)\$, \$F'(z)\$, and \$F'(z\000\1364)\$})%
{Algebraic independence of $F(z)$, $F(z^4)$, $F'(z)$, and $F'(z^4)$}
\label{s3}

In this section, we prove Theorem~\ref{funcs} up to a certain algebraic statement
concerning the nonexistence of polynomials satisfying a certain functional
equation. Because of possible independent interest, we provide a much more generalised version
of the statement than immediately needed for our current purpose. It is as follows.

\begin{theorem}
\label{th-tensor}
There are no polynomials $p_{m_0,\dots,m_s}(z)\in\mathbb{C}[z]$ \textup(besides all being trivial\textup) such that
\begin{multline}
\label{zer-tensor}
\lambda(z)\sum_{\substack{0\le m_j\le M_j\\j=0,1,\dots,s}}
p_{m_0,\dots,m_s}(z)y_0^{m_0}\dotsb y_s^{m_s}
\\[-15pt]
=\sum_{\substack{0\le m_j\le M_j\\j=0,1,\dots,s}}p_{m_0,\dots,m_s}(z^4)\prod_{i=0}^s(1+z+z^2-zy_i)^{M_i-m_i}
\end{multline}
for some rational function $\lambda(z)$.
\end{theorem}

This nonexistence result is proved in the next section.

To start our proof of Theorem~\ref{funcs}, we show that Theorem \ref{Fasymptotics} gives a recipe for computing
the radial asymptotics of $F(\xi z)$ as $z\to1^-$ for any root of unity $\xi$ of degree $4^n$.
Consider, for example, $\xi_1\in\{\pm i,-1\}$ and substitute $z=\xi_1z$ into
equation \eqref{Fdefn}. Using the asymptotics provided in Theorem~\ref{Fasymptotics}, then as $z\to1^-$ we have
\begin{align*}
F(\xi_1z)
&=(1+\xi_1z+\xi_1^2z^2)F(z^4)-z^4F(z^{16})
\\
&=(1+\xi_1+\xi_1^2)\frac{C(z^4)}{(1-z^4)^{\lg\rho}}(1+O(1-z^4))
-\frac{C(z^{16})}{(1-z^{16})^{\lg\rho}}(1+O(1-z^{16}))
\\
&=(1+\xi_1+\xi_1^2)\frac{C(z)}{(4(1-z))^{\lg\rho}}(1+O(1-z))
-\frac{C(z)}{(16(1-z))^{\lg\rho}}(1+O(1-z))
\\
&=\biggl((1+\xi_1+\xi_1^2)\frac{3-\sqrt5}2-\frac{7-3\sqrt5}2\biggr)
\frac{C(z)}{(1-z)^{\lg\rho}}(1+O(1-z))
\\
&=\Omega(\xi_1)\frac{C(z)}{(1-z)^{\lg\rho}}(1+O(1-z)),
\end{align*}
because $4^{-\lg\rho}=(3-\sqrt5)/2$. Similarly, if $\xi_2^4=\xi_1$ then as $z\to
1^-$ we have
\begin{equation*}
F(\xi_2z)
=(1+\xi_2z+\xi_2^2z^2)F(\xi_1z^4)-\xi_1z^4F(z^{16})
=\Omega(\xi_2)\frac{C(z)}{(1-z)^{\lg\rho}}(1+O(1-z)),
\end{equation*}
where
$$
\Omega(\xi_2)=(1+\xi_2+\xi_2^2)\Omega(\xi_1)\frac{3-\sqrt5}2-\xi_1\frac{7-3\sqrt5}2,
$$
and in general this iteration defines the function $\Omega(\xi)$
at any root of unity $\xi$ of degree $4^n$ for $n\ge 0$:
$$
\Omega(\xi)=(1+\xi+\xi^2)\Omega(\xi^4)\frac{3-\sqrt5}2-\xi^4\Omega(\xi^{16})\frac{7-3\sqrt5}2,
$$
and $\Omega(1)=1$, together with the related radial asymptotics of $F(\xi z)$.
Note that $(3-\sqrt5)/2=\rho^{-2}$ and $(7-3\sqrt5)/2=\rho^{-4}$.

\begin{lemma}
\label{lem1}
Let $\xi$ be a root of unity of degree $4^n$ for some $n\ge 0$. Then as $z\to1^-$, we have
\begin{equation*}
F(\xi z)=\Omega(\xi)\frac{C(z)}{(1-z)^{\lg\rho}}(1+O(1-z)),
\end{equation*}
where the function $\Omega(z)$ satisfies $\Omega(1)=1$ and
\begin{equation}
\Omega(z)=(1+z+z^2)\rho^{-2}\Omega(z^4)-z^4\rho^{-4}\Omega(z^{16}).
\label{AF2}
\end{equation}
\end{lemma}

We stress that the function $\Omega(z)$ and so its relative
\begin{equation}
\omega(z):=\frac{\rho^2\Omega(z)}{\Omega(z^4)}
\label{ome}
\end{equation}
are only defined on the set of roots of unity of degree $4^n$ where $n=0,1,2,\dots$\,.

\begin{lemma}
\label{lem1a}
The function \eqref{ome} is transcendental over the field of rational functions.
\end{lemma}

\begin{proof}
We start observing that the functional equation \eqref{AF2} translates into
\begin{equation}
\omega(z)=1+z+z^2-\frac{z^4}{\omega(z^4)}
\label{ACF}
\end{equation}
for the function~\eqref{ome}. Assume, on the contrary, that the function $\omega(z)$ is algebraic,
hence satisfies, on the set of the roots of unity, a non-trivial relation
$$
\sum_{m=0}^Mp_m(z)(\omega(z)/z)^m=\sum_{m=0}^Mp_m(z)y^m\big|_{y=\omega(z)/z}=0,
$$
which we suppose to have the least possible $M$. Multiply the relation by $(z/\omega(z))^M$,
substitute $z^4$ for $z$ in the relation, and apply \eqref{ACF} to arrive at
$$
\sum_{m=0}^Mp_m(z^4)(1+z+z^2-zy)^{M-m}\big|_{y=\omega(z)/z}
=\sum_{m=0}^Mp_m(z^4)(1+z+z^2-\omega(z))^{M-m}=0.
$$
If the two algebraic relations are not proportional, that is, if
$\sum_{m=0}^Mp_m(z^4)(1+z+z^2-zy)^{M-m}$ is not $\lambda(z)\sum_{m=0}^Mp_m(z)y^m$ for some $\lambda(z)\in\mathbb C(z)$,
then a suitable linear combination of the two will eliminate the term $y^M$ and result
in a non-trivial algebraic relation for $\omega(z)$ of degree smaller than~$M$, a contradiction.
The proportionality, on the other hand, is not possible in view of Theorem~\ref{th-tensor} applied in the case $s=0$.
Thus, $\omega(z)$ cannot satisfy an algebraic relation over $\mathbb C(z)$.
\end{proof}

Recall the function $\mu(z)$ defined in~\eqref{eq:mudef}.
It follows from \eqref{eq:murec} and \eqref{ACF} and from $\mu(1)=\rho^2=\omega(1)$ that the functions
$\mu(z)$ and $\omega(z)$ coincide on the set of roots of unity of degree $4^n$. Thus, Lemma~\ref{lem1a}
implies that $\mu(z)$ is a transcendental function\,---\,the fact which is already a consequence of the algebraic independence of
$F(z)$ and $F(z^4)$. Note, however, that in the opposite direction the transcendence of $\mu(z)$ does not directly imply Lemma~\ref{lem1a},
because the function $\omega(z)$ is defined on a smaller set of certain roots of unity and is not even known to be analytic.

Another immediate consequence of the transcendence of $\omega(z)$ is the following result.

\begin{lemma}
\label{lem2}
Assume that with polynomials $p_0(z),\dots,p_M(z)\in\mathbb C[z]$ we have
$$
\sum_{m=0}^Mp_m(\xi)\Omega(\xi)^m\Omega(\xi^4)^{M-m}=0
$$
for any root of unity of degree $4^n$, where $n=0,1,2,\dots$\,.
Then $p_m(z)=0$ for each $m=0,1,\dots,M$.
\end{lemma}

\begin{proof}
Indeed, the equation from the hypothesis of the lemma is equivalent to the identity
$$
\sum_{m=0}^Mp_m(z)\rho^{-2m}\omega(z)^m=0
$$
on the set of roots of unity of degree $4^n$ where $n=0,1,2,\dots$\,.
This contradicts the transcendence of $\omega(z)$ established in Lemma~\ref{lem1a}.
\end{proof}

Denoting
$$
F_k(z):=\biggl(z\,\frac{\d}{\d z}\biggr)^kF(z)
$$
and using the fact that $C(z)$ is real analytic,
we have, as $z\to 1^-$, that
$$
F_k(\xi z)=\prod_{j=0}^{k-1}(\lg\rho+j)\cdot\Omega(\xi)\frac{C(z)}{(1-z)^{\lg\rho+k}}(1+O(1-z)),
$$
since this is true for $k=0$ and we simply differentiate it as many times as needed. From now on,
we can consider the limit as $z\to1^-$ along the sequence $\exp(-t_04^{-n})$ for integers $n\ge1$, for
some fixed $t_0$, so that $C(z)$ is constant along the sequence.
Finally, we can write the functional equation for the derivatives in the form
\begin{equation}
F_k(z)=\bigl(4^k(1+z+z^2)F_k(z^4)-16^kz^4F_k(z^{16})\bigr)\cdot(1+o(1))
\label{ARF}
\end{equation}
as $z$ approaches any root of unity of degree $4^n$, because the terms in $o(1)$
involve the derivatives of $F$ of order smaller than~$k$.
(In fact, we will only use \eqref{ARF} for $k=1$.)

We are now in a position to present the proof of Theorem~\ref{funcs}.

\begin{proof}[Proof of Theorem \textup{\ref{funcs}}]
For the sake of a contradiction, assume that the theorem is false and that we have an algebraic relation
\begin{equation}
\sum_{\bm=(m_0,m_1,m_2,m_3)\in\bM}p_{\bm}(z)F_0(z)^{m_0}F_1(z)^{m_1}F_0(z^4)^{m_2}F_1(z^4)^{m_3}=0,
\label{SF}
\end{equation}
where the set $\bM$ of multi-indices $\bm\in\mathbb Z_{\ge 0}^4$
is finite and none of the polynomials $p_{\bm}(z)$ in the sum is identically zero.
Without loss of generality, we can assume that the polynomial
$\sum_{\bm}p_{\bm}(z)y_0^{m_0}y_1^{m_1}y_2^{m_2}y_3^{m_3}$ in \emph{five} variables is irreducible.

In the first part of our proof, we discuss the algebraic independence of
$F_0(z),F_1(z)$, and $F_0(z^4)$ only (so that the dependence on $y_3$ is
suppressed); this scheme is general for this particular case as well as for the
one with $F_1(z^4)$.

Let $\xi$ be any root of unity of degree $4^n$ for some $n\ge0$. Note that as $z\to 1^-$, we have
\begin{equation*}
F_0(\xi z)^{m_0}F_1(\xi z)^{m_1}F_0\bigl((\xi z)^4\bigr)^{m_2}F_1\bigl((\xi z)^4\bigr)^{m_3}
=C_{\bm}\cdot\frac{\Omega(\xi)^{m_0+m_1}\Omega(\xi^4)^{m_2+m_3}}{(1-z)^{(\lg\rho)|\bm|+(m_1+m_3)}}(1+o(1))
\end{equation*}
where $|\bm|:=m_0+m_1+m_2+m_3$,
$$
C_{\bm}:=\frac{C^{|\bm|}}{4^{m_3}}\biggl(\frac{3-\sqrt5}2\biggr)^{m_2+m_3}(\lg\rho)^{m_1+m_3},
$$
and $C=C(e^{-t_0/4})$ does not depend on $\xi$ or $z$, the latter chosen along the sequence.

Denote by $\bM'$ the subset of all multi-indices of $\bM$ for which the quantity
$$
\beta:=(\lg\rho)|\bm|+m_1+m_3
$$
is maximal; in particular, $|\bm|$ and $m_1+m_3$ are the same for all $\bm\in\bM'$.

Substituting $\xi z$ for $z$ in \eqref{SF}, multiplying all the terms in the resulted sum by $(1-z)^\beta$, and letting $z\to1^-$, we deduce that
\begin{equation}
\sum_{\bm\in\bM'}C_{\bm}\cdot p_{\bm}(\xi)\cdot\Omega(\xi)^{m_0+m_1}\Omega(\xi^4)^{m_2+m_3}=0
\label{SF1}
\end{equation}
for any root of unity $\xi$ under consideration.
If there is no dependence on $F_1(z^4)$ in \eqref{SF} (hence in \eqref{SF1}) then
the summation in $m_3$ is suppressed;
in this case $M:=|\bm|=m_0+m_1+m_2$ and $M':=m_1$ are constant for all indices $\bm\in\bM'$,
so that equation \eqref{SF1} becomes
\begin{equation*}
\sum_{\bm=(m_0,M',M-M'-m_0,0)\in\bM'}C_{\bm}\cdot p_{\bm}(\xi)\cdot\Omega(\xi)^{m_0}\Omega(\xi^4)^{M-m_0}=0
\end{equation*}
for any root of unity $\xi$ of degree $4^n$.
By Lemma~\ref{lem2}, this is only possible when $p_{\bm}(z)=0$ identically, a contradiction to our choice of~$\bM$.
This means that the functions $F_0(z)$, $F_1(z)$ and $F_0(z^4)$ are algebraically independent.

The same argument in the case of general \eqref{SF} implies that
\begin{equation*}
\sum_{\bm=(m_0,N-m_0,M_0-m_0,M_1-(N-m_0))\in\bM'}C_{\bm}\cdot p_{\bm}(z)=0
\end{equation*}
for any $N$, where $M_0:=m_0+m_2$ and $M_1:=m_1+m_3$ are constant on~$\bM'$.

We next iterate relation \eqref{SF} and compute, again, the asymptotics of the
leading term as $z$ tends radially to a root of unity $\xi$ of degree~$4^n$.
For  this, we substitute $z^4$ for $z$ in the relation \eqref{SF}, multiply the
result by $16^{m_1+m_3}z^{4|\bm|}$ and use the expressions
for $F_0(z^{16})$ and $F_1(z^{16})$ given by the Mahler functional equations in
$F_0(z)$, $F_0(z^4)$, $F_1(z)$, and $F_1(z^4)$
(see also the proof of Theorem~\ref{numbers} below)
to arrive at
\begin{equation*}
\sum_{\bn\in\bN}q_{\bn}(z)F_0(z)^{n_0}F_1(z)^{n_1}F_0(z^4)^{n_2}F_1(z^4)^{n_3}=0,
\end{equation*}
where $\bN$ is defined analogous to $\bM$. The terms contributing the
leading asymptotics (which, of course, remains attached to the same $\beta$ as
before) correspond to the multi-indices
$\bn\in\bN'$ with the property $\beta=(\lg\rho)|\bn|+n_1+n_3$. Because of
\eqref{ARF}, controlling the coefficients
$q_{\bn}(z)$ for $\bn\in\bN'$ is much easier than for general $\bn\in\bN$.
Note that $\bN'$ is characterised by constant
$M_0=n_0+n_2$ and $M_1=n_1+n_3$ as was before $\bM'$.
The above transformation for the \emph{leading asymptotics terms} in~\eqref{SF}
assumes the form
\begin{align*}
\sum_{\bm\in\bM'}
& 16^{m_1}z^{4(m_0+m_1)}p_{\bm}(z^4)F_0(z^4)^{m_0}\bigl((1+z+z^2)F_0(z^4)-F_0(z)\bigr)^{m_2}
\\ &\qquad\times
F_1(z^4)^{m_1}\bigl(4(1+z+z^2)F_1(z^4)-F_1(z)\bigr)^{m_3}
\\
&=\sum_{\bm\in\bM'}16^{m_1}z^{4(m_0+m_1)}p_{\bm}(z^4)
\\ &\qquad\times
\sum_{n_0=0}^{m_2}(-1)^{n_0}\binom{m_2}{n_0}(1+z+z^2)^{m_2-n_0}
F_0(z)^{n_0}F_0(z^4)^{m_2+m_0-n_0}
\\ &\qquad\times
\sum_{n_1=0}^{m_3}(-1)^{n_1}\binom{m_3}{n_1}4^{m_3-n_1}(1+z+z^2)^{m_3-n_1}F_1(z)^{n_1}F_1(z^4)^{m_3+m_1-n_1}
\end{align*}
implying
\begin{align*}
q_{\bn}(z)
&=(-1)^{n_0+n_1}\sum_{\substack{n_0\le m_2\le n_0+n_2\\n_1\le m_3\le n_1+n_3}}
16^{n_1+n_3-m_3}\,4^{m_3-n_1}\binom{m_2}{n_0}\binom{m_3}{n_1}
\\ &\quad\times
z^{4(|\bn|-m_2-m_3)}(1+z+z^2)^{m_2+m_3-n_0-n_1}p_{(n_0+n_2-m_2,n_1+n_3-m_3,m_2,m_3)}(z^4)
\displaybreak[2]\\
&=(-1)^{n_0+n_1}\sum_{\substack{n_0\le m_2\le M_0\\n_1\le m_3\le M_1}}
16^{M_1}\,4^{-m_3-n_1}\binom{m_2}{n_0}\binom{m_3}{n_1}
\\ &\quad\times
z^{4(M_0-m_2+M_1-m_3)}(1+z+z^2)^{m_2+m_3-n_0-n_1}p_{(M_0-m_2,M_1-m_3,m_2,m_3)}(z^4)
\\
&=(-1)^{n_0+n_1}16^{M_1}\sum_{\substack{0\le m_0\le M_0-n_0\\0\le m_1\le M_1-n_1}}
4^{m_1-n_1-M_1}\binom{M_0-m_0}{n_0}\binom{M_1-m_1}{n_1}
\\ &\quad\times
z^{4(m_0+m_1)}(1+z+z^2)^{M_0+M_1-m_0-m_1-n_0-n_1}p_{(m_0,m_1,M_0-m_0,M_1-m_1)}(z^4)
\end{align*}
for all $\bn\in\bN'$, where for the last equality we switched to summation
over $m_0=M_0-m_2$ and $m_1=M_1-m_3$.

In view of our assumption of the irreducibility of the original algebraic relation \eqref{SF},
the newer relation must be proportional to it; that is, the polynomial
$\sum_{\bm\in\bM}p_{\bm}(z)y_0^{m_0}y_1^{m_1}y_2^{m_2}y_3^{m_3}$ divides
$\sum_{\bn\in\bN}q_{\bn}(z)y_0^{n_0}y_1^{n_1}y_2^{n_2}y_3^{n_3}$ in the polynomial ring $\mathbb C[z,y_0,y_1,y_2,y_3]$.
In particular, the leading asymptotic parts of these polynomials,
$$
\sum_{\bn\in\bN'}q_{\bn}(z)y_0^{n_0}y_1^{n_1}y_2^{n_2}y_3^{n_3}
=y_2^{M_0}y_3^{M_1}\sum_{\bn\in\bN'}q_{\bn}(z)\biggl(\frac{y_0}{y_2}\biggr)^{n_0}\biggl(\frac{y_1}{y_3}\biggr)^{n_1}
$$
and
$$
\sum_{\bm\in\bM'}p_{\bm}(z)y_0^{m_0}y_1^{m_1}y_2^{m_2}y_3^{m_3}
=y_2^{M_0}y_3^{M_1}\sum_{\bm\in\bM'}p_{\bm}(z)\biggl(\frac{y_0}{y_2}\biggr)^{m_0}\biggl(\frac{y_1}{y_3}\biggr)^{m_1},
$$
must be proportional, hence their quotient must be a polynomial in $z$. In other words,
the sets $\bN'$ and $\bM'$ coincide (unless the former is empty, meaning
that $q_{\bm}(z)=0$ identically for all $\bm\in\bM'$)
and $q_{\bm}(z)=q(z)p_{\bm}(z)$ for all $\bm\in\bM'$ for some $q(z)\in\mathbb C[z]$.
We define
$$
\wh p_{m_0,m_1}(z):=4^{m_1}p_{m_0,m_1,M_0-m_0,M_1-m_1}(z)
\quad\text{for}\; 0\le m_0\le M_0, \; 0\le m_1\le M_1,
$$
so that
$$
\sum_{m_0=0}^{M_0}\sum_{m_1=0}^{M_1}p_{(m_0,m_1,M_0-m_0,M_1-m_1)}(z)y_0^{m_0}y_1^{m_1}
=\sum_{m_0=0}^{M_0}\sum_{m_1=0}^{M_1}\wh p_{m_0,m_1}(z)y_0^{m_0}\biggl(\frac{y_1}4\biggr)^{m_1},
$$
and the above proportionality relation reads
\begin{multline*}
q(z)\sum_{m_0=0}^{M_0}\sum_{m_1=0}^{M_1}\wh p_{m_0,m_1}(z)y_0^{m_0}y_1^{m_1}
\\
=4^{M_1}\sum_{m_0=0}^{M_0}\sum_{m_1=0}^{M_1}\wh p_{m_0,m_1}(z^4)(1+z+z^2-zy_0)^{M_0-m_0}(1+z+z^2-zy_1)^{M_1-m_1}.
\end{multline*}
However, it follows from the case $s=1$ of Theorem~\ref{th-tensor} that this is not possible;
that is, the polynomials must all be identically zero.
\end{proof}

In this final part of the section, we prove Theorem~\ref{numbers}, by applying Theorem~\ref{funcs}
along with a general result in Mahler's method due to Ku.~Nishioka \cite{N1990}; see also her monograph~\cite{N1996},
in particular, Theorem~4.2.1 there.

\begin{proposition}[Ku.~Nishioka \cite{N1990}]
\label{Nishioka}
Let $K$ be an algebraic number field and let $k\ge 2$ be a positive integer.
Let $f_1(z),\dots,f_d(z)\in K[[z]]$ and write $\bbf(z)$ for the column-vector $(f_1(z),\ldots,f_d(z))^T$.
If
$$
\bbf(z^k)=\bB(z)\bbf(z)
$$
for some matrix $\bB(z)\in K(z)^{d\times d}$ and $\alpha$ is a nonzero algebraic number in the radius of convergence of $\bbf(z)$
such that $\alpha^{k^j}$ is not a pole of $\bB(z)$ for any $j\ge 0$,
then
$$
\trdeg_{\mathbb{Q}}\mathbb{Q}(f_1(\alpha),\dots,f_d(\alpha))\ge \trdeg_{K(z)}K(z)(f_1(z),\dots,f_d(z)).
$$
\end{proposition}

\begin{proof}[Proof of Theorem~\textup{\ref{numbers}}]
Apply Theorem \ref{funcs} and Proposition~\ref{Nishioka} with $K=\mathbb{Q}$, $k=d=4$,
$$
\bbf(z)= (F(z), F(z^{4}),F'(z),F'(z^{4}))^T,
$$
and
\begin{equation*}
\bB(z)=\begin{pmatrix}
0 & 1 & 0 & 0 \\
-\frac{1}{z^4} & \frac{1+z+z^2}{z^4} & 0 & 0 \\
0 & 0 & 0 & 1 \\
\frac{1}{4z^{20}} & -\frac{4+3z+2z^2}{16z^{20}} & -\frac{1}{16z^{19}} & \frac{1+z+z^2}{4z^{16}}
\end{pmatrix}.
\qedhere
\end{equation*}
\end{proof}

\section{Linear algebra and Fibonacci numbers}
\label{s4}

This section is entirely devoted to the proof of Theorem~\ref{th-tensor}.

Assume, on the contrary, that a non-trivial collection of polynomials $p_{m_0,\dots,m_s}(z)$ satisfying \eqref{zer-tensor} exists.
If the greatest common divisor of the polynomials is $p(z)$ then dividing them all by $p(z)$ we arrive at the relation \eqref{zer-tensor}
for the newer normalised polynomials but with $\lambda(z)$ replaced by $\lambda(z)p(z)/p(z^4)$.
Therefore, we can assume without loss of generality that the polynomials $p_{m_0,\dots,m_s}(z)$ in \eqref{zer-tensor} are relatively prime.
Furthermore, without loss of generality we can assume that the existing polynomials all have \emph{rational} coefficients
as the identity \eqref{zer-tensor} itself happens to be over the field of rationals,
so that $p_{m_0,\dots,m_s}(z)\in\mathbb{Q}[z]$ and $\lambda(z)\in\mathbb Q(z)$.

We first analyse the $s=0$ case of relation~\eqref{zer-tensor}. To this end, suppose there exist
$p_0(z),\dots,p_M(z)\in\mathbb{Q}[z]$ with $\gcd(p_0(z),\dots,p_M(z))=1$ such that
\begin{equation}
\label{zer}
\lambda(z)\sum_{m=0}^M p_m(z)y^m
=\sum_{m=0}^M p_m(z^4)(1+z+z^2-zy)^{M-m}
\end{equation}
for some rational function $\lambda(z)$. Assuming $\lambda(z)$ is nonzero, write
$\lambda(z)=a(z)/b(z)$, where $\gcd(a(z),b(z))=1$, so that \eqref{zer} becomes
\begin{equation}
\label{ab}
a(z)\sum_{m=0}^M p_m(z)y^m
=b(z)\sum_{m=0}^M p_m(z^4)(1+z+z^2-zy)^{M-m}.
\end{equation}
It follows immediately that any polynomial $p_m(z)$ on the left-hand side of
\eqref{ab} is divisible by $b(z)$, hence $b:=b(z)$ is a constant.
By substituting $x=1+z+z^2-zy$, we write \eqref{ab} as
\begin{equation*}
a(z)\sum_{m=0}^M p_m(z)z^{M-m}(1+z+z^2-x)^m
=bz^M\sum_{m=0}^M p_m(z^4)x^{M-m},
\end{equation*}
from which we conclude as before that each $p_m(z^4)$ is divisible by $a(z)/z^N$
where $z^N$ is the highest power of $z$ dividing $a(z)$.
As $\gcd(p_0(z^4),\dots,p_M(z^4))=1$ we find out that $a:=a(z)/z^N$ is a constant.
In summary, $\lambda(z)=\lambda z^N$ for some $\lambda\in\mathbb Q$ and $N\in\mathbb Z_{\ge 0}$; that is,
\begin{equation}
\label{lam}
\lambda z^N\sum_{m=0}^M p_m(z)y^m
=\sum_{m=0}^M p_m(z^4)(1+z+z^2-zy)^{M-m}.
\end{equation}
Note that the rational constant $\lambda$ must be nonzero
as otherwise, by substituting $z=1$ into~\eqref{lam}, all $p_m(z^4)$ would have the common divisor $z-1$.

The above argument clearly extends to tensor powers of the related operator.
In this way, we use the above to extend to the general case in the statement of the theorem.

\begin{lemma}
\label{lam-lem}
Assuming relation \eqref{zer-tensor} holds, we have $\lambda(z)=\lambda z^N$ for
some $\lambda\in\mathbb Q\setminus\{0\}$ and $N\in\mathbb Z_{\ge 0}$.
\end{lemma}

In our further investigation we will be interested in specialising identity \eqref{zer-tensor}
by choosing $z$ to be an appropriate root of unity. Any such specialisation leads to
a linear relation on the space of polynomials in $y_0,y_1,\dots,y_s$ of degree at most
$M_j$ in $y_j$ for each $j=0,1,\dots,s$.

Take a matrix
$$
\gamma:=\begin{pmatrix} a & b \\ c & d \end{pmatrix}\in\SL_2(\mathbb C),
$$
so that its determinant $ad-bc=1$,
and assume that its eigenvalues $\mu$ and $\mu^{-1}$ are distinct.
Consider the linear operator $$U_M(\gamma)\colon y^m\mapsto(ay+b)^m(cy+d)^{M-m}$$
on the linear space $\cP_M[y]$ of polynomials of degree at most~$M$.
Let the two row vectors $(\alpha_0,\beta_0)$ and $(\alpha_1,\beta_1)$ be the
eigenvectors of $\gamma$; that is,
$$
(\alpha_0,\beta_0)\gamma=\mu(\alpha_0,\beta_0)
\quad\text{and}\quad
(\alpha_1,\beta_1)\gamma=\mu^{-1}(\alpha_1,\beta_1).
$$

\begin{lemma}
\label{det-lem}
The spectrum of $U_M(\gamma)$ is the set
$$
\{\mu^k:-M\le k\le M,k\equiv M\;(\operatorname{mod}2)\},
$$
with the corresponding eigenpolynomials
\begin{gather*}
r_k(y)=r_k(\gamma;y)=(\alpha_0y+\beta_0)^{(M+k)/2}(\alpha_1y+\beta_1)^{(M-k)/2},
\\
-M\le k\le M, \quad k\equiv M\;(\operatorname{mod}2).
\end{gather*}
\end{lemma}

\begin{proof}
This follows immediately from the fact that the operator $U_1(\gamma)$ maps
$\alpha_0y+\beta_0$ onto $\mu(\alpha_0y+\beta_0)$
and $\alpha_1y+\beta_1$ onto $\mu^{-1}(\alpha_1y+\beta_1)$.
\end{proof}

Lemma \ref{det-lem} allows us to describe the spectrum of the (tensor-product)
operator
$$
U=U_{M_0}(\gamma)\otimes\dots\otimes U_{M_s}(\gamma)
\colon\prod_{j=0}^sy_j^{m_j}\mapsto\prod_{j=0}^s(ay_j+b)^{m_j}(cy_j+d)^{M_j-m_j}
$$
that acts on the space of polynomials in $y_0,y_1,\dots,y_s$ of degree at most
$M_j$ in $y_j$ for $j=0,1,\dots,s$,
as well as to explicitly produce the corresponding eigenpolynomials.

\begin{lemma}
\label{spec-lem}
The spectrum of $U$ is contained in $\mu^{\mathbb Z}$,
with the \textup(linearly independent\textup) eigenpolynomials
$$
\prod_{j=0}^sr_{k_j}(\gamma;y_j),
\qquad -M_j\le k_j\le M_j, \quad k_j\equiv M_j\;(\operatorname{mod}2)
\quad\text{for}\; j=0,1,\dots,s,
$$
corresponding to the eigenvalues $\mu^{k_0+k_1+\dots+k_s}$, respectively.
\end{lemma}

Substitution $z=1$ into \eqref{zer-tensor}, where $\lambda(z)=\lambda z^N$, brings our case to
\begin{align*}
&
\lambda\sum_{\substack{0\le m_j\le M_j\\j=0,1,\dots,s}}
p_{m_0,\dots,m_s}(1)y_0^{m_0}\dotsb y_s^{m_s}
\\ &\qquad
=\sum_{\substack{0\le m_j\le M_j\\j=0,1,\dots,s}}p_{m_0,\dots,m_s}(1)
(3-y_0)^{M_0-m_0}\dotsb(3-y_s)^{M_s-m_s}.
\end{align*}
This corresponds to an eigenvector $r(y_0,\dots,y_s)$ of the operator $U$
when $\gamma=\bigl(\begin{smallmatrix} 0 & 1 \\ -1 & 3 \end{smallmatrix}\bigr)\in\SL_2(\mathbb C)$.
We get $\mu=(3+\sqrt5)/2=\rho^2$,
$$
\alpha_0y+\beta_0=\mu^{-1}y-1 \quad\text{and}\quad \alpha_1y+\beta_1=\mu y-1.
$$
(Of course, we exclude the trivial case $r(y_0,\dots,y_s)=0$ as it would imply
that all $p_{m_0,\dots,m_s}(z)$ are divisible by $z-1$.)
Because all $p_{m_0,\dots,m_s}(z)\in\mathbb Q[z]$, we have $r(y_0,\dots,y_s)\in
\mathbb Q[y_0,\dots,y_s]$, so that
none of the irrational values in $\mu^{\mathbb Z}$ can show up as $\lambda$. In
other words, $\lambda=1$
and the structure of the tensor product above dictates $M_0+\dots+M_s$ to be even
and further produces
\begin{equation}
\label{eigen}
r(y_0,\dots,y_s)
=\sum_{\substack{|k_j|\le M_j, \; k_j\equiv M_j\;(\operatorname{mod}2)\\
k_0+\dots+k_s=0}}C_{k_0,\dots,k_s}\cdot r_{k_0}(y_0)\dotsb r_{k_s}(y_s)
\end{equation}
where $C_{k_0,\dots,k_s}\in\mathbb Q[\mu]$. We do not require the form
\eqref{eigen}, but only the fact $\lambda=1$.

\begin{lemma}
\label{sharp-lem}
Assuming relation \eqref{zer-tensor} holds, we have $\lambda(z)=z^N$ for some $N\in\mathbb Z_{\ge 0}$\textup:
\begin{multline}
\label{lam-z00}
z^N\sum_{\substack{0\le m_j\le M_j\\j=0,1,\dots,s}}p_{m_0,\dots,m_s}(z)y_0^{m_0}\dotsb y_s^{m_s}
\\[-15pt]
=\sum_{\substack{0\le m_j\le M_j\\j=0,1,\dots,s}}p_{m_0,\dots,m_s}(z^4)\prod_{i=0}^s(1+z+z^2-zy_i)^{M_i-m_i}.
\end{multline}
\end{lemma}

We now take any prime $p>3$, a root of unity $\zeta_p$ of degree $p$, and the matrix
\begin{gather}
\label{gammap}
\gamma_p:=g(\zeta_p^{4^{(p-1)/2-1}})g(\zeta_p^{4^{(p-1)/2-2}})\dotsb g(\zeta_p^4)g(\zeta_p),
\\
\text{where}\quad
g(z)=\begin{pmatrix} 0 & 1 \\ -z & z^2+z+1 \end{pmatrix}.
\nonumber
\end{gather}
If we write $\gamma_p=\bigl(\begin{smallmatrix} a & b \\ c & d \end{smallmatrix}\bigr)$, iterate the right-hand side of~\eqref{lam-z00},
and substitute $z=\zeta_p$, then we obtain
\begin{multline}
\label{lam-z10}
\zeta_p^N\sum_{\substack{0\le m_j\le M_j\\j=0,1,\dots,s}}
p_{m_0,\dots,m_s}(\zeta_p)y_0^{m_0}\dotsb y_s^{m_s}
\\[-15pt]
=\sum_{\substack{0\le m_j\le M_j\\j=0,1,\dots,s}}
p_{m_0,\dots,m_s}(\zeta_p)\prod_{i=0}^s (ay_i+b)^{m_i}(cy_i+d)^{M_i-m_i},
\end{multline}
because $\zeta_p^{4^{(p-1)/2}}=\zeta_p^{2^{p-1}}=\zeta_p$. Note that $\det g(z)=z$, so that
$$
\det\gamma_p=\prod_{j=0}^{(p-1)/2-1}\zeta_p^{4^j}=\zeta_p^{(2^{p-1}-1)/3}=1
$$
for primes $p>3$, thus establishing that $\gamma_p\in\SL_2(\mathbb C)$.

\begin{lemma}
\label{p-lem}
Let $\mu$ and $\mu^{-1}$ be the eigenvalues of $\gamma_p$.
If $\mu^p\ne1$ then $N\equiv0\pmod p$ in \eqref{lam-z00}.
\end{lemma}

\begin{proof}
Comparing relation \eqref{lam-z10} with the result of Lemma~\ref{spec-lem} we conclude
that $\zeta_p^N\in\mu^{\mathbb Z}$, and the latter is only possible when $\zeta_p^N=1$.
\end{proof}

\begin{lemma}
\label{inf-lem}
There are infinitely many primes $p$ for
which the eigenvalues $\mu$ of the corresponding $\gamma_p$
are not $p$-th roots of unity.
\end{lemma}

We have checked by direct computation that the only primes $p$ in the range $3<p<300$, for
which the condition $\mu^p\ne1$ is violated, are $p=5$ and $p=11$ (and $\mu=1$ in the two cases).
It is therefore natural to expect that we \emph{always} have the condition of Lemma~\ref{p-lem} satisfied
for primes $p>11$.

\begin{proof}[Proof of Lemma~\textup{\ref{inf-lem}}]
For positive exponents $e_1,\dots,e_s$ consider
$$
g(z^{e_1})g(z^{e_2})\dotsb g(z^{e_s})
=\begin{pmatrix} -a(z) & b(z) \\ -c(z) & d(z) \end{pmatrix}.
$$
Using induction on $s$, the coefficients of polynomials $a(z),b(z),c(z),d(z)$ are nonnegative integers;
furthermore,
$$
\begin{pmatrix} -a(1) & b(1) \\ -c(1) & d(1) \end{pmatrix}
=g(1)^s
=\begin{pmatrix} -F_0 & F_2 \\ -F_2 & F_4 \end{pmatrix}^s
=\begin{pmatrix} -F_{2s-2} & F_{2s} \\ -F_{2s} & F_{2s+2} \end{pmatrix},
$$
where $F_0=0$, $F_1=1$, and $F_n=F_{n-1}+F_{n-2}$ is the Fibonacci sequence. It
is not hard to verify that
\begin{equation}
\label{fib1}
F_{2n}\equiv\begin{cases}
0\;(\operatorname{mod}7) & \text{if $n\equiv0\;(\operatorname{mod}4)$}, \\
3\;(\operatorname{mod}7) & \text{if $n\equiv2\;(\operatorname{mod}8)$}, \\
4\;(\operatorname{mod}7) & \text{if $n\equiv6\;(\operatorname{mod}8)$},
\end{cases}
\qquad n=0,2,4,\dotsc.
\end{equation}

Consider now any prime $p\equiv15\;(\operatorname{mod}28)$; this means that $p-1$ is
divisible by~7 and that $s=(p-1)/2$, the number of the matrices $g(\,\cdot\,)$ in
the product for~$\gamma_p$, is odd. The trace of the matrix $\gamma_p$ is a
polynomial in $\zeta_p$ containing $F_{2s+2}$ monomials $\zeta_p^m$ minus
$F_{2s-2}$ monomials $\zeta_p^m$.
Note the the sums of the form $\sum_{j=1}^{p-1}\zeta_p^{jm}=0$, each involving $p-1$ terms,
for $j\not\equiv0\;(\operatorname{mod}p)$, and such sums only, can be cancelled from consideration,
thus leaving us with
$$
F_{2(s+1)}-F_{2(s-1)}-(p-1)N\equiv\pm3\;(\operatorname{mod}7)
$$
terms according to~\eqref{fib1}.
This implies that the trace of~$\gamma_p$ in its irreducible form is the sum of
\emph{at least three} monomials $\zeta_p^m$, corresponding to not necessarily
different $m\in\{0,1,\dots,p-1\}$; in particular, the trace cannot be written
in the form $\zeta_p^\ell+\zeta_p^{-\ell}$ for some~$\ell$. Therefore, the
eigenvalues $\mu,\mu^{-1}$ of $\gamma_p$ are not of the form $\zeta_p^{\ell},\zeta_p^{-\ell}$.
This completes the proof of Lemma~\ref{inf-lem}.
\end{proof}

\begin{proof}[Proof of Theorem~\textup{\ref{th-tensor}}]
It follows from Lemmas~\ref{p-lem} and \ref{inf-lem} that $N\equiv0\pmod p$ in \eqref{lam-z00} for infinitely many primes~$p$.
This means that $N=0$ and equation~\eqref{lam-z00} assumes the form
\begin{multline}
\label{lam-z000}
\sum_{\substack{0\le m_j\le M_j\\j=0,1,\dots,s}}p_{m_0,\dots,m_s}(z)y_0^{m_0}\dotsb y_s^{m_s}
\\[-15pt]
=\sum_{\substack{0\le m_j\le M_j\\j=0,1,\dots,s}}p_{m_0,\dots,m_s}(z^4)\prod_{i=0}^s(1+z+z^2-zy_i)^{M_i-m_i}.
\end{multline}
As $y_0,\dots,y_s$ are replaced with $c_0y_0,\dots,c_sy_s$, where $(c_0,\dots,c_s)$ varies over $\mathbb R^{s+1}$,
the generic degree in $z$ of the right-hand side in \eqref{lam-z000} is bounded
from below by $4d$, where $d$~denotes the maximal degree of the polynomials $p_{m_0,\dots,m_s}(z)$, while
the degree in~$z$ of the left-hand side in \eqref{lam-z000} is at most~$d$.
This can only happen when the polynomials are constant; however in the latter
circumstances we will still have a positive degree in~$z$
for the right-hand side in \eqref{lam-z000}, a contradiction completing the proof
of Theorem~\ref{th-tensor}.
\end{proof}

\section{Mahler functions at roots of unity}
\label{s5}

In this section, we discuss the structure of general Mahler functions at roots of unity and
provide	an alternative approach to the proof of Theorem~\ref{Fasymptotics}, which can be used in the
asymptotical study of Mahler functions of any degree.

The simplest possible Mahler functions are given as infinite products, so that it is natural to investigate
the asymptotics of
$$
P(z):=\prod_{j=0}^\infty\frac1{(1-\alpha z^{k^j})}
$$
as $z\to1^-$, where $\alpha\in\mathbb C$, $|\alpha|\le 1$. The recent
paper \cite{ABklMahler} provides crude estimates for the asymptotics of such products,
though earlier works \cite{dB1948,DF1996,M1940} already discuss the asymptotics in the `most natural' case $|\alpha|=1$;
see also the paper \cite{DukeNguyen}.

As in Section~\ref{sec:asymp}, we make use of the Mellin transform, and so we define
$$
\mathcal{P}(s):=\int_0^\infty\ln P(e^{-t})\,t^{s-1}\d t,
$$
which maps $e^{-\lambda t}$ to $\Gamma(s)\lambda^{-s}$.
Since
$$
\ln P(e^{-t})
=\sum_{j=0}^\infty\sum_{l=1}^\infty\frac1l\alpha^le^{-lk^jt},$$
we have
\begin{equation}
\label{MTP}
\mathcal{P}(s)=\Gamma(s)\sum_{j=0}^\infty\sum_{l=1}^\infty\frac1l\,\frac{\alpha^l}{(lk^j)^s}
=\frac{\Gamma(s)}{1-k^{-s}}\sum_{l=1}^\infty\frac{\alpha^l}{l^{s+1}}.
\end{equation}
Thus the asymptotics of $\ln P(e^{-t})$ as $t\to0^+$ is related to the values of
the meromorphic continuation of the Dirichlet series
$$
\frac{\mathcal{P}(s)}{\Gamma(s)}=\frac1{1-k^{-s}}\sum_{l=1}^\infty\frac{\alpha^l}{l^{s+1}}
$$
at negative integers \cite[Proposition~2]{Z2006}.
Without reproducing the standard analytical argument in this situation (see \cite{dB1948,DF1996} for details)
one gets, as $t\to 0^+$,
$$
P(e^{-t})=C(t)t^{(\ln(1-\alpha))/(\ln k)}(1+O(t))
$$
if $\alpha\ne1$, and
$$
P(e^{-t})=C(t)t^{-1/2}e^{(\ln^2t)/(2\ln k)}(1+O(t))
$$
if $\alpha=1$, for some positive and $(2\pi im/\ln k)$-periodic function $C(t)$ of~$t$.
Clearly, the asymptotics so obtained allow one to
write out the asymptotic behaviour of any solution of the Mahler equation
$$
f(z)=\prod_{j=1}^m(1-\xi_jz)\cdot f(z^k)
$$
along any radial limit as $z\to\xi$, where $\xi_1,\dots,\xi_m$ and $\xi$ are roots of unity.

This analysis, Theorem~\ref{Fasymptotics} and the approach we discuss below allow us
to expect similar asymptotic behaviour for other Mahler functions $f(z)$
satisfying functional equations of the form~\eqref{eq-mahler}.
That is, under some natural conditions imposed on the polynomials $a(z),a_0(z),\dots,a_d(z)$, the asymptotics of $f(z)$ as $z\to1^-$
is either of the form
$$
C(z)(1-z)^{c_1}(1+O(1-z))
$$
or
$$
C(z)(1-z)^{c_0}e^{c_2\ln^2(1-z)}(1+O(1-z))
$$
for some rational $c_0$ or $c_1$ of the form `$\ln(\text{an algebraic integer})/\ln k$' and $c_2$ of the form
`$\text{rational}/\ln k$', where the function $C(z)$ is assumed to have some oscillatory behaviour.

For our alternative method to prove Theorem~\ref{Fasymptotics}, we consider the function $F(z)$ as defined in
the introduction. If we let $z=e^{-t}$ where $t=4^{-x}$ and denote
$f(x)=F(z^{16})=F(e^{-16t})$, then the functional equation \eqref{Fdefn} assumes
the form
$$
f(x+2)-(1+e^{-t}+e^{-2t})f(x+1)+e^{-4t}f(x)=0.
$$
Using $|1-e^{-t}|<t$ for $t>0$, we can then recast this equation in the form
$$
f(x+2)-(3+a_1(x))f(x+1)+(1+a_2(x))f(x)=0,
$$
where $|a_1(x)|,|a_2(x)|\le 4\cdot4^{-x}$. Denoting the zeroes of the
characteristic polynomial $\lambda^2-3\lambda+1$, by $\lambda_1:=(3-\sqrt5)/2$
and $\lambda_2:=(3+\sqrt5)/2=\rho^2$, and applying the quantitative version of
Perron's theorem due to Coffman \cite{C1964} (see \cite[Theorem~2]{T1993} and
comments to it within for the explicit statement, as well as \cite{E1953} and
\cite{GK1953} for the predecessors), we deduce that
$$
f(x)=\tilde C\lambda_j^x(1+O(4^{-x}))
$$
as $x\to+\infty$ along $x\equiv x_0\pmod{\mathbb Z}$,
for some $\tilde C=\tilde C(x_0)>0$ and $j\in\{1,2\}$. A simple analysis then shows that $j=2$.

Note that $\tilde C(x)$ is a $1$-periodic real-analytic function in an interval $x>\sigma_0$
because of the analytic dependence of the solution of the difference equation on the initial data.
This implies that as $t\to0^+$
$$
F(e^{-t})=\frac{\hat C(t)}{t^{\lg\rho}}(1+O(t)),
$$
and further
\begin{equation*}
F(z)=\frac{C(z)}{(1-z)^{\ln\rho}}(1+O(1-z))
\qquad\text{as}\quad z\to1^-,
\end{equation*}
where $C(z)$ is real-analytic and satisfies $C(z)=C(z^4)$ for $z\in(0,1)$,
which is exactly the statement of Theorem~\ref{Fasymptotics}.

The Mellin-transform approach of Section~\ref{sec:asymp} and the difference-equation approach of this section
also make possible studying the asymptotic behaviour of Mahler functions at other roots of unity.
Here we briefly explain the situation on what happens with the particular example of $F(z)$.

Denote $\zeta_n$ a primitive root of unity of (odd) degree $n$.
For $\zeta_3$ we clearly have $\zeta_3^4=\zeta_3$, therefore the defining
equation \eqref{Fdefn} for $F(z)$ transforms to the equation
$$
\tilde F(z)=(1+\zeta_3z+\zeta_3^2z^2)\tilde F(z^4)-\zeta_3z^4\tilde F(z^{16})
$$
for the function $\tilde F(z):=F(\zeta_3z)$. The characteristic polynomial of
this recursion as $z\to1^-$ is $\lambda^2+\zeta_3$,
with the absolute values of both roots equal to 1, so that the theorem of Poincar\'e \cite{G1971}
applies to imply that $\tilde F(z)$ has oscillatory behaviour as $z\to1^-$.

Similarly, the difference equation for $F(z)$ gives rise to a difference equation
for $\tilde F(z)=F(\zeta_nz)$, because the former is equivalent
to a relation between $F(z)$, $F(z^{4^k})$ and $F(z^{4^{2k}})$ for any $k\ge 1$.
In particular, one can take $k=(p-1)/2$ when $n=p$ is a prime
(compare with the construction of $\gamma_p$ in Section~\ref{s4}).
Taking $n=5$ and choosing $k=2$, so that $\zeta_n^{4^k}=\zeta_n$,
we find the corresponding characteristic polynomial $\lambda^2-2\lambda+1$.
Again, the double zero $\lambda=1$ (of absolute value~1)
of the polynomial leads to the oscillatory behaviour of $\tilde F(z)$ as $z\to1^-$;
the same story happens for the choice $n=11$ and $k=6$.
The first `interesting' situation originates at $n=7$. Here $k=3$ and the
characteristic polynomial of the difference equation relating
$F(\zeta_7z)$, $F(\zeta_7z^{4^3})$ and $F(\zeta_7z^{4^6})$ is
$$
\lambda^2+(\zeta_7+\zeta_7^2+\zeta_7^4)\lambda+1=\lambda^2+\frac{-1\pm\sqrt{-7}}2\lambda+1.
$$
Its roots are
$$
\frac{1\pm\sqrt{-7}\pm\sqrt{-22\pm\sqrt{-7}}}4,
$$
whose absolute values are approximately $0.53101005$ and $1.88320350$.
The technical conditions of Coffman's version of Perron's theorem are not met in this case,
but the numerics support that they behave
$$
F(\zeta_7z)\sim C(z)(1-z)^{\ln(0.53101005)/(3\ln4)}
\qquad\text{as}\quad z\to1^-,
$$
where $C(z)$ oscillates.

The characteristic polynomial for $n=p$ is exactly the characteristic polynomial
of the matrix $\gamma_p$ in \eqref{gammap}. It may be interesting to look for the
known $L$-functions as potential Mellin transforms for such a sophisticated
behaviour at different roots of unity: recall that the Mellin transform \eqref{MTP}
used at the beginning of this section was, up to the unwanted factor $(1-k^{-s})^{-1}$, a
Dirichlet $L$-function. Though it is hard to expect anything significant
for the Mellin transform of $\ln F(z)$ as there are zeroes of $F(z)$ in the disc $|z|<1$
(compare with the footnote in Section~\ref{sec:asymp}), but $F(z)$ itself
looks a nice target for a reasonable $L$-function.

\begin{acknowledgements}
We thank the referees of the journal for their valuable feedback that helped us to improve presentation of the paper.
\end{acknowledgements}

\bibliographystyle{amsplain}
\def\cprime{$'$}
\providecommand{\bysame}{\leavevmode\hbox to3em{\hrulefill}\thinspace}
\providecommand{\MR}{\relax\ifhmode\unskip\space\fi MR }
\providecommand{\MRhref}[2]{%
  \href{http://www.ams.org/mathscinet-getitem?mr=#1}{#2}
}
\providecommand{\href}[2]{#2}


\end{document}